\documentclass[10pt,reqno]{amsart}
\usepackage[hypertex]{hyperref}

\usepackage{amsmath,amsthm}
\usepackage{amssymb}
\usepackage{euscript}

\numberwithin{equation}{subsection}

\newcommand{\sqsp}{\renewcommand{\baselinestretch}{1.15}\tiny\normalsize}

\newcommand{\relphantom[1]}{\mathrel{\phantom{#1}}}
\newtheorem{theorem}[subsection]{Theorem}
\newtheorem{lemma}[subsection]{Lemma}
\newtheorem{proposition}[subsection]{Proposition}
\newtheorem{corollary}[subsection]{Corollary}

\theoremstyle{definition}
\newtheorem{definition}[subsection]{Definition}
\newtheorem{example}[subsection]{Example}
\newtheorem{remark}[subsection]{Remark}

\newcommand{\bk}{\mathbf{k}}
\newcommand{\bZ}{\mathbf{Z}}
\newcommand{\oct}{\mathbf{O}}

\newcommand{\xbar}{\overline{x}}
\newcommand{\ybar}{\overline{y}}

\newcommand{\muop}{\mu^{op}}
\newcommand{\cyclicsum}{{\circlearrowright\,}}
\newcommand{\bracket}{[,...\,,]}

\DeclareMathOperator{\Real}{Re}

\begin{document}

\title{On $n$-ary Hom-Nambu and Hom-Maltsev algebras}
\author{Donald Yau}

\begin{abstract}
Hom-alternative and Hom-Jordan algebras are shown to give rise to Hom-Nambu algebras of arities $2^{k+1} + 1$.  The class of $n$-ary Hom-Maltsev algebras is studied.  Multiplicative $n$-ary Hom-Nambu-Lie algebras are shown to be $n$-ary Hom-Maltsev algebras. Examples of ternary Hom-Maltsev algebras that are not ternary Hom-Nambu-Lie algebras are given.  Ternary Hom-Maltsev algebras are shown to arise from composition algebras.
\end{abstract}



\address{Department of Mathematics\\
    The Ohio State University at Newark\\
    1179 University Drive\\
    Newark, OH 43055, USA}
\email{dyau@math.ohio-state.edu}

\date{\today}
\maketitle

\sqsp

%%%%%%%%%%%%%%%%%%%%%%
\section{Introduction}
%%%%%%%%%%%%%%%%%%%%%%

Algebras with $n$-ary compositions of Lie and Jordan types are important in Lie and Jordan theories, geometry, analysis, and physics.  Most simple Lie algebras can be obtained from Jordan triple systems through the Kantor-Koecher-Tits construction \cite{kantor1,koecher,tits}.  On the other hand, Lie triple systems give rise to $\bZ/2\bZ$-graded Lie algebras \cite{jacobson1,lister}, and they are exactly the kind of Lie algebras associated to symmetric spaces.  In geometry and analysis, Jordan triple systems are used in the classification of symmetric spaces \cite{bertram,chu,kaup1,kaup2,loos1,loos2}.

In physics $n$-ary algebras appear in a number of different contexts.  For instance, Nambu mechanics \cite{nambu} involves $n$-ary Nambu algebras, in which the $n$-ary composition is a derivation.  This derivation property, called the $n$-ary Nambu identity, generalizes the Jacobi identity in Lie algebras.  In particular, Lie triple systems are ternary Nambu algebras with two further properties.  Lie triple systems can be used to construct solutions of the Yang-Baxter equation \cite{okubo}.  Superconformal algebras can be realized by Jordan triple systems \cite{gun,gh,gh2}.  An $n$-ary Nambu algebra whose product is anti-symmetric is called an $n$-ary Nambu-Lie algebra.   Ternary Nambu-Lie algebras with some extra structures appear in the work of Bagger and Lambert on $M$-branes \cite{bl}.  Other applications of $n$-ary algebras in physics are discussed in, e.g., \cite{bg1,bg2,kerner2,plet}.

% n-ary Hom-algebras of Lie and associative types
Hom-type generalizations of $n$-ary Nambu(-Lie) algebras, called $n$-ary Hom-Nambu(-Lie) algebras, were introduced by Ataguema, Makhlouf, and Silvestrov in \cite{ams} .  Each $n$-ary Hom-Nambu(-Lie) algebra has $n-1$ linear twisting maps, which appear in a twisted generalization of the $n$-ary Nambu identity called the $n$-ary Hom-Nambu identity (see Definition \ref{def:homnambulie}).  If the twisting maps are all equal to the identity, one recovers an $n$-ary Nambu(-Lie) algebra.  The twisting maps provide a substantial amount of freedom in manipulating  Nambu(-Lie) algebras.  For example, in \cite{ams} it is demonstrated that some ternary Nambu-Lie algebras can be regarded as ternary Hom-Nambu-Lie algebras with non-identity twisting maps.

% origin in Hom-Lie, other recent work
% Hom-algebras

Binary Hom-Nambu-Lie algebras, usually called Hom-Lie algebras, originated in \cite{hls} in the study of $q$-deformations of the Witt and the Virasoro algebras.  Hom-associative algebra was defined in \cite{ms}.  Hom-Lie algebras are to Hom-associative algebras as Lie algebras are to associative algebras \cite{ms,yau}.  Other Hom-type algebras can be found in \cite{mak} - \cite{ms4} and \cite{yau3,yau4,yau12,yau13}.  Hom-quantum groups and the Hom-Yang-Baxter equations have been studied by the author in \cite{yau5} - \cite{yau11}.

Let us recall several known constructions of $n$-ary Hom-Nambu(-Lie) algebras.  In \cite{ams} it is shown that $n$-ary Nambu(-Lie) algebras can be twisted along self-morphisms to yield $n$-ary Hom-Nambu(-Lie) algebras.  This twisting construction is a generalization of a result about $G$-Hom-associative algebras due to the author \cite{yau2}.  In \cite{ams1} it is shown that ternary Virasoro-Witt algebras \cite{cfz} can be $q$-deformed into ternary Hom-Nambu-Lie algebras.  In \cite{ams2} it is shown that a ternary Hom-Nambu-Lie algebra can be obtained from a Hom-Lie algebra together with a compatible linear map and a trace function.

Some further properties of $n$-ary Hom-Nambu(-Lie) algebras were established by the author in \cite{yau13}.  A unique feature of Hom-type algebras is that they are closed under twisting by suitably defined self-morphisms.  For example, it is shown in \cite{yau13} that the category of $n$-ary Hom-Nambu(-Lie) algebras is closed under twisting by self-weak morphisms.  If one begins with an $n$-ary Nambu(-Lie) algebra, then this closure property reduces to the twisting construction for $n$-ary Hom-Nambu(-Lie) algebras in \cite{ams}.  Therefore, conceptually one can regard the category of $n$-ary Hom-Nambu(-Lie) algebras as an extension of the category of $n$-ary Nambu(-Lie) algebras that has this closure property.

Furthermore, it is proved in \cite{yau13} that every multiplicative $n$-ary Hom-Nambu algebra yields a sequence of Hom-Nambu algebras of exponentially higher arities (see Theorem \ref{thm:higher}).  Going in the other direction, it is shown in \cite{yau13} that, under suitable conditions, an $n$-ary Hom-Nambu(-Lie) algebra reduces to an $(n-1)$-ary Hom-Nambu(-Lie) algebra.  Hom-Jordan and Hom-Lie triple systems were also defined in \cite{yau13}, which generalize Jordan and Lie triple systems \cite{jacobson1,lister,meyberg}.  A Hom-Lie triple system is automatically a ternary Hom-Nambu algebra, but it is usually not a ternary Hom-Nambu-Lie algebra because its ternary product is not assumed to be anti-symmetric.  It is proved in \cite{yau13} that Hom-Lie triple systems, and hence ternary Hom-Nambu algebras, can be obtained from Hom-Jordan triple systems (see Theorem \ref{thm:jtslts}), ternary totally Hom-associative algebras \cite{ams}, multiplicative Hom-Lie algebras, and Hom-associative algebras.

% purpose

There are two main purposes in this paper.  First, we show that Hom-alternative and Hom-Jordan algebras yield Hom-Nambu algebras of arities $2^{k+1} + 1$.   In fact, we show that multiplicative Hom-Jordan algebras have underlying Hom-Jordan triple systems, generalizing a well-known fact about Jordan algebras and Jordan triple systems.  Then we combine it with results from \cite{yau12,yau13} to realize Hom-Nambu algebras using Hom-alternative and Hom-Jordan algebras.  These results provide a large class of Hom-Nambu algebras that are usually not Hom-Nambu-Lie algebras.  Second, we introduce and study $n$-ary Hom-Maltsev algebras, which include multiplicative $n$-ary Hom-Nambu-Lie algebras.  A description of the rest of this paper follows.

In section \ref{sec:prelim} we recall some properties of Hom-Jordan algebras and Hom-Jordan triple systems.  In particular, we recall from \cite{yau12} that multiplicative Hom-alternative algebras are Hom-Jordan admissible (Theorem \ref{thm:homalt}).  Also we recall from \cite{yau13} that Hom-Jordan triple systems give rise to Hom-Lie triple systems through the Meyberg construction (Theorem \ref{thm:jtslts}).

It is well-known that Jordan algebras have underlying Jordan triple systems (Corollary \ref{cor1:hj}).  In section \ref{sec:homjordan} we show that there is a Hom-type analogue of this fact, i.e., that multiplicative Hom-Jordan algebras have underlying multiplicative Hom-Jordan triple systems (Theorem \ref{thm:hjhjts}).  The proof of this result is an intricate calculation involving the left multiplication operators in a Hom-Jordan algebra.

It is observed in section \ref{sec:homn} that multiplicative Hom-Jordan algebras have underlying multiplicative Hom-Lie triple systems, and hence multiplicative ternary Hom-Nambu algebras (Corollary \ref{cor2:hj}). Furthermore, every multiplicative Hom-Jordan algebra gives rise to a sequence of multiplicative Hom-Nambu algebras of arities $2^{k+1}+1$ (Corollary \ref{cor1:higher}).  The author showed in \cite{yau12} that multiplicative Hom-alternative algebras are Hom-Jordan admissible.  It follows that multiplicative Hom-alternative algebras also yield multiplicative Hom-Jordan triple systems (Corollary \ref{cor3:hj}) and multiplicative Hom-Nambu algebras of arities $2^{k+1}+1$ (Corollaries \ref{cor4:hj} and \ref{cor2:higher}).

In section \ref{sec:maltsev} we define $n$-ary Hom-Maltsev algebras. They include Pozhidaev's $n$-ary Maltsev algebras \cite{poz} (when the twisting maps are all equal to the identity) and the author's Hom-Maltsev algebras \cite{yau12} (when $n=2$).  Maltsev algebras were introduced by Maltsev \cite{maltsev}, who called them Moufang-Lie algebras.  Lie algebras are examples of Maltsev algebras, but there are plenty of non-Lie Maltsev algebras \cite{myung}.  Maltsev algebras are important in the geometry of smooth loops \cite{kerdman,kuzmin,maltsev,nagy,sabinin}.  We show that the category of $n$-ary Hom-Maltsev algebras is closed under twisting by self-morphisms (Theorem \ref{thm1:hommaltsev}).  Then we show that ternary Hom-Maltsev algebras can be constructed from composition algebras with involution (Corollary \ref{cor3:hommaltsev}).  As an example, we show that the octonion algebra carries many ternary Hom-Maltsev algebra structures that are not ternary Hom-Nambu-Lie algebras (Example \ref{ex:octonion}).

Two further properties of $n$-ary Hom-Maltsev algebras are established in section \ref{sec:homnambu}.  These results are Hom-type generalizations of Pozhidaev's results \cite{poz} about $n$-ary Maltsev algebras.  First,  we show that every multiplicative $n$-ary Hom-Nambu-Lie algebra is also an $n$-ary Hom-Maltsev algebra (Theorem \ref{thm:liemaltsev}).   Second, we show that under some conditions $n$-ary Hom-Maltsev algebras give rise to $(n-k)$-ary Hom-Maltsev algebras (Theorem \ref{thm2:reduction}).  Using this result, we show how Hom-Maltsev algebras arise from composition algebras with involution (Corollaries \ref{cor2:reduction} and \ref{cor3:reduction}).

%%%%%%%%%%%%%%%%%%%%%%%%%%%%%%%%%%%%%%%%%%%%%%%%%%%%%%%%%%%
\section{Hom-Jordan algebras and Hom-Jordan triple systems}
\label{sec:prelim}
%%%%%%%%%%%%%%%%%%%%%%%%%%%%%%%%%%%%%%%%%%%%%%%%%%%%%%%%%%%

The purpose of this section is to recall some basic properties of Hom-Jordan algebras \cite{yau12} and Hom-Jordan triple systems \cite{yau13}.

%%%%%%%%%%%%%%%%%%%%%%%%
\subsection{Conventions}

Throughout this paper we work over a fixed field $\bk$ of characteristic $0$.  The dimension of a $\bk$-module $V$ is denoted by $\dim(V)$.  If $V$ is a $\bk$-module and $f \colon V \to V$ is a linear map, then $f^n$ denotes the composition of $n$ copies of $f$ with $f^0 = Id$.  For $i \leq j$, elements $x_i,\ldots,x_j \in V$ and maps $f,f_k,\ldots,f_l \colon V \to V$ with $j-i=l-k$, we use the abbreviations
\begin{equation}
\label{xij}
\begin{split}
x_{i,j} &= (x_i,x_{i+1},\ldots,x_j),\\
f(x_{i,j}) &= (f(x_i), f(x_{i+1}),\ldots, f(x_j)),\\
f_{k,l}(x_{i,j}) &= (f_k(x_i),f_{k+1}(x_{i+1}),\ldots,f_l(x_j)).
\end{split}
\end{equation}
For $i > j$, the symbols $x_{i,j}$, $f(x_{i,j})$, and $f_{k,l}(x_{i,j})$ denote the empty sequence.  For a bilinear map $\mu \colon V^{\otimes 2} \to V$, we often write $\mu(x,y)$ as the juxtaposition $xy$.  Denote by $\muop$ the opposite map, $\muop(x,y) = \mu(y,x)$.

% n-ary Hom-algebras, anti-symmetry, multiplicativity, (weak) morphisms

Let us begin with the following basic definitions.

% definitions
\begin{definition}
\label{def:nhomalgebra}
Let $n \geq 2$ be an integer.
\begin{enumerate}
\item
An \textbf{$n$-ary Hom-algebra} $(V,\bracket,\alpha)$ with $\alpha=(\alpha_1,\ldots,\alpha_{n-1})$ \cite{ams} consists of a $\bk$-module $V$, an $n$-linear map $\bracket \colon V^{\otimes n} \to V$, and linear maps $\alpha_i \colon V \to V$ for $i = 1, \ldots , n-1$, called the \textbf{twisting maps}.
\item
An $n$-ary Hom-algebra $(V,\bracket,\alpha)$ is said to be \textbf{multiplicative} if the twisting maps are all equal, i.e., $\alpha_1 = \cdots = \alpha_{n-1} = \alpha$, and $\alpha \circ \bracket = \bracket \circ \alpha^{\otimes n}$.
\item
A \textbf{weak morphism} $f \colon V \to U$ of $n$-ary Hom-algebras is a linear map of the underlying $\bk$-modules such that $f \circ \bracket_V = \bracket_U \circ f^{\otimes n}$.  A \textbf{morphism} of $n$-ary Hom-algebras is a weak morphism such that $f \circ (\alpha_i)_V = (\alpha_i)_U \circ f$ for $i = 1, \ldots n-1$ \cite{ams}.
\item
The $n$-ary product $\bracket$ in an $n$-ary Hom-algebra $V$ is said to be \textbf{anti-symmetric} if
\[
[x_1,\ldots,x_n] = \epsilon(\sigma)[x_{\sigma(1)},\ldots,x_{\sigma(n)}]
\]
for all $x_i \in V$ and permutations $\sigma$ on $n$ letters, where $\epsilon(\sigma)$ is the signature of $\sigma$.
\end{enumerate}
\end{definition}

An $n$-ary Hom-algebra $V$ in which all the twisting maps are equal, as in the multiplicative case, will be denoted by $(V,\bracket,\alpha)$, where $\alpha$ is the common value of the twisting maps.  An $n$-ary Hom-algebra is called \textbf{binary} or \textbf{ternary} if $n=2$ or $3$, respectively.  An \textbf{$n$-ary algebra} in the usual sense is a $\bk$-module $V$ with an $n$-linear map $\bracket \colon V^{\otimes n} \to V$.  We consider an $n$-ary algebra $(V,\bracket)$ also as an $n$-ary Hom-algebra $(V,\bracket,Id)$ in which all $n-1$ twisting maps are the identity map.  Also, in this case a weak morphism is the same thing as a morphism, which agrees with the usual definition of a morphism of $n$-ary algebras.

% def: Hom-algebras, Hom-associator, Hom-Jordan, Hom-alternative algebras

Next we recall the Hom-type generalizations of associative, alternative, Jordan, and flexible algebras.

\begin{definition}
\label{def:homalg}
\begin{enumerate}
\item
A \textbf{Hom-algebra} is a binary Hom-algebra $(A,\mu,\alpha)$.
\item
The \textbf{Hom-associator} of a Hom-algebra $(A,\mu,\alpha)$ \cite{ms} is the trilinear map $as_A \colon A^{\otimes 3} \to A$ defined as
\begin{equation}
\label{homassociator}
as_A = \mu \circ (\mu \otimes \alpha - \alpha \otimes \mu).
\end{equation}
A \textbf{Hom-associative algebra} \cite{ms} is a Hom-algebra whose Hom-associator is equal to $0$.
\item
A \textbf{Hom-alternative algebra} \cite{mak} is a Hom-algebra whose Hom-associator is anti-symmetric.
\item
A \textbf{Hom-Jordan algebra} \cite{yau12} is a Hom-algebra $(A,\mu,\alpha)$ such that $\mu = \muop$ (commutativity) and the \textbf{Hom-Jordan identity}
\begin{equation}
\label{homjordanid}
as_A(x^2,\alpha(y),\alpha(x)) = 0
\end{equation}
is satisfied for all $x,y \in A$, where $x^2 = \mu(x,x)$.
\item
A \textbf{Hom-flexible algebra} \cite{ms} is a Hom-algebra $(A,\mu,\alpha)$ that satisfies
\[
as_A(x,y,x) = 0
\]
for all $x,y \in A$.
\end{enumerate}
\end{definition}

It follows immediately from the definitions that Hom-associative algebras are also Hom-alternative algebras, which are themselves Hom-flexible algebras.  Also, the commutativity of the multiplication implies that Hom-Jordan algebras are also Hom-flexible algebras.

Construction results and examples of Hom-associative, Hom-alternative, and Hom-Jordan algebras can be found in \cite{yau2}, \cite{mak}, and \cite{yau12}, respectively.

\begin{remark}
In \cite{mak} Makhlouf defined a Hom-Jordan algebra as a commutative Hom-algebra satisfying $as_A(x^2,y,\alpha(x)) = 0$, which becomes our Hom-Jordan identity \eqref{homjordanid} if $y$ is replaced by $\alpha(y)$. Using Makhlouf's definition of a Hom-Jordan algebra, Hom-alternative algebras are not Hom-Jordan admissible, although Hom-associative algebras are still Hom-Jordan admissible \cite{mak}.
\end{remark}

It is well-known that alternative algebras are Jordan admissible \cite{schafer}.  Using the definition in \eqref{homjordanid}, the Hom-version is also true.  In fact, multiplicative Hom-alternative algebras are Hom-Jordan admissible in the following sense.

\begin{theorem}[\cite{yau12}]
\label{thm:homalt}
Let $(A,\mu,\alpha)$ be a multiplicative Hom-alternative algebra.  Then $A$ is Hom-Jordan admissible, in the sense that the Hom-algebra
\[
A^+ = (A,\ast,\alpha)
\]
is a multiplicative Hom-Jordan algebra, where $\ast = (\mu + \muop)/2$.
\end{theorem}

The Hom-algebra $A^+ = (A,\ast,\alpha)$ in Theorem \ref{thm:homalt} is called the \textbf{plus Hom-algebra} of $A$.  Note that $x^2$ is the same whether we use $\mu$ or the Jordan product $\ast$, since $\mu(x,x) = x \ast x$.

Actually, a multiplicative Hom-alternative algebra is also Hom-Maltsev admissible, in the sense that the commutator Hom-algebra
\[
A^- = (A,[,]=\mu - \muop,\alpha)
\]
is a Hom-Maltsev algebra \cite{yau12}.  However, we do not need this result in this paper.  Higher arity generalizations of Hom-Maltsev algebras will be discussed in later sections.

Using Hom-flexibility and commutativity, the Hom-Jordan identity \eqref{homjordanid} can be linearized in the following way.

\begin{proposition}[\cite{yau12}]
\label{prop:homjordanid}
Let $(A,\mu,\alpha)$ be a Hom-Jordan algebra.  Then
\begin{equation}
\label{homjordanid2}
\cyclicsum_{x,w,z} \,as_A(\alpha(x),\alpha(y),wz) = 0
\end{equation}
for all $w,x,y,z \in A$, where $\cyclicsum_{x,w,z}$ denotes the cyclic sum over $(x,w,z)$.
\end{proposition}

We call \eqref{homjordanid2} the \textbf{linearized Hom-Jordan identity}.

Next we recall some ternary Hom-algebras from \cite{yau13}.

% Hom-Jordan to Hom-Lie triple systems (Meyberg construction)

\begin{definition}
\label{def:hjts}
\begin{enumerate}
\item
A \textbf{Hom-triple system} is a ternary Hom-algebra $(V,\{,,\},\alpha=(\alpha_1,\alpha_2))$.
\item
A \textbf{Hom-Jordan triple system} is a Hom-triple system $(J,\{,,\},\alpha)$ that satisfies
\begin{equation}
\label{outersymmetry}
\{xyz\} = \{zyx\} \quad\text{(outer-symmetry)}
\end{equation}
and the \textbf{Hom-Jordan triple identity}
\begin{equation}
\label{homjtsid}
\begin{split}
\{\alpha_1(x)\alpha_2(y)\{uvw\}\} &- \{\alpha_1(u)\alpha_2(v)\{xyw\}\}\\
&= \{\{xyu\}\alpha_1(v)\alpha_2(w)\} - \{\alpha_1(u)\{yxv\}\alpha_2(w)\}
\end{split}
\end{equation}
for all $u,v,w,x,y \in J$.
\item
A \textbf{Hom-Lie triple system} is a Hom-triple system $(L,[,,],\alpha)$ that satisfies
\begin{equation}
\label{homltsid}
\begin{split}
[uvw] &= -[vuw] \quad\text{(left anti-symmetry)},\\
0 &= [uvw] + [wuv] + [vwu] \quad\text{(ternary Jacobi identity)},
\end{split}
\end{equation}
and the \textbf{ternary Hom-Nambu identity} \cite{ams}
\begin{equation}
\label{homnambu}
[\alpha_1(x)\alpha_2(y)[uvw]] = [[xyu]\alpha_1(v)\alpha_2(w)] + [\alpha_1(u)[xyv]\alpha_2(w)] + [\alpha_1(u)\alpha_2(v)[xyw]]
\end{equation}
for all $u,v,w,x,y \in J$.
\end{enumerate}
\end{definition}

When the twisting maps $\alpha_i$ are both equal to the identity map, we recover the usual notions of a \textbf{Jordan triple system} \cite{meyberg} and a \textbf{Lie triple system} \cite{jacobson1,lister}.  So Jordan and Lie triple systems are examples of multiplicative Hom-Jordan and Hom-Lie triple systems, respectively.  In the special case when both twisting maps are the identity map, we call \eqref{homjtsid} the \textbf{Jordan triple identity}.

Construction results and examples of Hom-Jordan and Hom-Lie triple systems can be found in \cite{yau13}.  In particular, Hom-Lie triple systems can be constructed from Hom-Jordan triple systems, ternary totally Hom-associative algebras \cite{ams}, Hom-associative algebras, and Hom-Lie algebras.

Generalizing an observation of Meyberg \cite{meyberg2} (XI Theorem I), one can show that Hom-Jordan triple systems give rise to Hom-Lie triple systems in the following way.

% hom-jts to hom-lts
\begin{theorem}[\cite{yau13}]
\label{thm:jtslts}
Let $(J,\{,,\},\alpha)$ be a Hom-Jordan triple system with equal twisting maps.  Define the triple product
\begin{equation}
\label{hjtshlts}
[xyz] = \{xyz\} - \{yxz\}
\end{equation}
for $x,y,z \in J$.  Then $L(J) = (J,[,,],\alpha)$ is a Hom-Lie triple system.  Moreover, if $J$ is multiplicative, then so is $L(J)$.
\end{theorem}

Hom-Lie triple systems are related to Hom-Nambu algebras \cite{ams}, which we now recall.

\begin{definition}
\label{def:homnambulie}
Let $(V,\bracket,\alpha=(\alpha_1,\ldots,\alpha_{n-1}))$ be an $n$-ary Hom-algebra.
\begin{enumerate}
\item
The \textbf{$n$-ary Hom-Jacobian} of $V$ is the $(2n-1)$-linear map $J^n_V \colon V^{\otimes 2n-1} \to V$ defined, using the shorthand in \eqref{xij}, as
\begin{equation}
\label{homjacobian}
\begin{split}
J^n_V(x_{1,n-1};y_{1,n})
&= [\alpha_{1,n-1}(x_{1,n-1}),[y_{1,n}]]\\
&\relphantom{} - \sum_{i=1}^n \,[\alpha_{1,i-1}(y_{1,i-1}), [x_{1,n-1},y_i], \alpha_{i,n-1}(y_{i+1,n})]
\end{split}
\end{equation}
for $x_1, \ldots, x_{n-1}, y_1, \ldots , y_n \in V$.
\item
An \textbf{$n$-ary Hom-Nambu algebra} \cite{ams} is an $n$-ary Hom-algebra $V$ that satisfies the \textbf{$n$-ary Hom-Nambu identity} $J^n_V = 0$.
\item
An \textbf{$n$-ary Hom-Nambu-Lie algebra} \cite{ams} is an $n$-ary Hom-Nambu algebra in which the $n$-ary product $\bracket$ is anti-symmetric.
\end{enumerate}
\end{definition}

When the twisting maps are all equal to the identity map, $n$-ary Hom-Nambu algebras and $n$-ary Hom-Nambu-Lie algebras are the usual \textbf{$n$-ary Nambu algebras} and \textbf{$n$-ary Nambu-Lie algebras} \cite{filippov,nambu,plet}.  In this case, the $n$-ary Hom-Jacobian $J^n_V$ is called the \textbf{$n$-ary Jacobian} \cite{filippov}, and the $n$-ary Hom-Nambu identity $J^n_V = 0$ is called the \textbf{$n$-ary Nambu identity}.  Construction results and examples of $n$-ary Hom-Nambu(-Lie) algebras can be found in \cite{ams1,ams2,ams,yau13}.

In terms of the ternary Hom-Jacobian \eqref{homjacobian}, the ternary Hom-Nambu identity \eqref{homnambu} is equivalent to $J^3_L = 0$.  Therefore, Hom-Lie triple systems are automatically ternary Hom-Nambu algebras.  Note that the ternary product in a Hom-Lie triple system is only assumed to be left anti-symmetric.  So Hom-Lie triple systems are in general not ternary Hom-Nambu-Lie algebras.

%%%%%%%%%%%%%%%%%%%%%%%%%%%%%%%%%%%%%%%%%%%%%%%%%%%%%%%%%%%
\section{From Hom-Jordan algebras to Hom-Jordan triple systems}
\label{sec:homjordan}
%%%%%%%%%%%%%%%%%%%%%%%%%%%%%%%%%%%%%%%%%%%%%%%%%%%%%%%%%%%

It is known that Jordan algebras give rise to Jordan triple systems.  The main result of this section is the following Hom-version of this fact.

% Main theorem
\begin{theorem}
\label{thm:hjhjts}
Let $(A,\mu,\alpha)$ be a multiplicative Hom-Jordan algebra.  Define the triple product
\begin{equation}
\label{hjtriple}
\{xyz\} = \alpha(x)(yz) + (xy)\alpha(z) - \alpha(y)(xz)
\end{equation}
for $x,y,z \in A$, where $\mu(x,y) = xy$.  Then
\[
A_T = (A,\{,,\},\alpha^2)
\]
is a multiplicative Hom-Jordan triple system.
\end{theorem}

Note that the twisting map of $A_T$ is $\alpha^2$, not $\alpha$.

The proof of Theorem \ref{thm:hjhjts} requires a series of Lemmas about Hom-Jordan algebras.  Let us introduce some notations.  In what follows, to simplify the typography, the composition of two maps will be written as $fg$ rather than $f \circ g$.

\begin{definition}
\label{def:L}
Let $(A,\mu,\alpha)$ be a Hom-algebra.  For $w,x,y,z \in A$, define the linear maps $L(x)$, $L(x,y)$, $L(x,y,z)$, $L(w,x,y,z)$, and $L_{xy}$ on $A$ by:
\[
\begin{split}
L(x)(w) &= xw,\\
L(x,y) &= L(\alpha(x))L(y) - L(\alpha(y))L(x),\\
L(x,y,z) &= L(\alpha(x),\alpha(y))L(z) - L(\alpha^2(z))L(x,y),\\
L(w,x,y,z) &= L(\alpha(w),\alpha(x),\alpha(y))L(z) - L(\alpha^3(z))L(w,x,y),\\
L_{xy} &= L(xy) \alpha + L(x,y).
\end{split}
\]
\end{definition}

In other words, $L(x)$ is the operator of left multiplication by $x$.  To interpret $L(x,y)$, $L(x,y,z)$, and $L(w,x,y,z)$, consider the special case $\alpha = Id$.  In this case, we have
\[
\begin{split}
L(x,y) &= L(x)L(y) - L(y)L(x) = [L(x),L(y)],\\
L(x,y,z) &= [L(x,y),L(z)] = [[L(x),L(y)],L(z)],\\
L(w,x,y,z) &= [L(w,x,y),L(z)] = [[[L(w),L(x)],L(y)],L(z)],
\end{split}
\]
where $[,]$ denotes the commutator bracket.  So in general $L(x,y)$, $L(x,y,z)$, and $L(w,x,y,z)$ are Hom-type analogues of the commutator bracket $[L(x),L(y)]$ and the iterated brackets $[[L(x),L(y)],L(z)]$ and $[[[L(w),L(x)],L(y)],L(z)]$.  Finally, applying $L_{xy}$ to $z \in A$, we have
\begin{equation}
\label{lxy}
L_{xy}(z) = (xy)\alpha(z) + \alpha(x)(yz) - \alpha(y)(xz),
\end{equation}
which is the triple product $\{xyz\}$ in \eqref{hjtriple}.  In other words,
\[
L_{xy} = \{xy(-)\}
\]
is the left multiplication operator with respect to the triple product in \eqref{hjtriple}.

Now we consider some basic properties of these linear operators.  The next Lemma will be used below often without further comment.

\begin{lemma}
\label{lem1:hj}
Let $(A,\mu,\alpha)$ be a multiplicative Hom-algebra with $\mu$ commutative.  Then for all $x,y,z \in A$, we have:
\begin{enumerate}
\item
$L(x)(y) = L(y)(x)$.
\item
$L(x+y) = L(x) + L(y)$.
\item
$L(x,y) = - L(y,x)$.
\item
$\alpha L(x) = L(\alpha(x)) \alpha$.
\item
$\alpha L(x,y) = L(\alpha(x),\alpha(y)) \alpha$.
\item
$\alpha L(x,y,z) = L(\alpha(x),\alpha(y),\alpha(z)) \alpha$.
\end{enumerate}
\end{lemma}

\begin{proof}
All the assertions are immediate from Definition \ref{def:L}.
\end{proof}

We will need the following operator form of the linearized Hom-Jordan identity \eqref{homjordanid2}.

\begin{lemma}
\label{lem2:hj}
Let $(A,\mu,\alpha)$ be a multiplicative Hom-Jordan algebra.  Then
\[
\cyclicsum_{x,w,z}\, L(\alpha(x),wz)\alpha = 0
\]
for all $w,x,z \in A$, where $\cyclicsum_{x,w,z}$ denotes the cyclic sum over $x$, $w$, and $z$.
\end{lemma}

\begin{proof}
By the commutativity of $\mu$, the linearized Hom-Jordan identity \eqref{homjordanid2} can be rewritten as
\[
\begin{split}
0 &= \cyclicsum_{x,w,z} \left\{L(\alpha(wz))L(\alpha(x)) - L(\alpha^2(x))L(wz)\right\}\alpha(y)\\
&= \cyclicsum_{x,w,z}\, L(wz,\alpha(x))\alpha(y)\\
&= - \cyclicsum_{x,w,z}\, L(\alpha(x),wz)\alpha(y).
\end{split}
\]
This proves the Lemma.
\end{proof}

In a Jordan algebra, it is known that the commutator $[L(x),L(y)]$ of two left multiplication operators is a derivation \cite{schafer} (Chapter IV).  One way to express this derivative property is
\[
L([L(x),L(y)]z) = [[L(x),L(y)],L(z)].
\]
The following Lemma is the Hom-version of this fact.

\begin{lemma}
\label{lem3:hj}
Let $(A,\mu,\alpha)$ be a multiplicative Hom-Jordan algebra.  Then
\begin{equation}
\label{lxyz2}
L(L(x,y)z)\alpha^2 = L(x,y,z)
\end{equation}
for all $x,y,z \in A$.
\end{lemma}

\begin{proof}
First observe that by definition,
\begin{equation}
\label{lxyz}
\begin{split}
L(x,y,z) &= L(\alpha^2(x))L(\alpha(y))L(z) - L(\alpha^2(y))L(\alpha(x))L(z)\\
&\relphantom{} - L(\alpha^2(z))L(\alpha(x))L(y) + L(\alpha^2(z))L(\alpha(y))L(x).
\end{split}
\end{equation}
Using the commutativity of $\mu$ and the multiplicativity of $\alpha$ and regarding the linearized Hom-Jordan identity \eqref{homjordanid2} as an operation on $w$, we can rewrite it as:
\begin{equation}
\label{star'}
\begin{split}
0 &= L(\alpha(xy))L(\alpha(z))\alpha - L(\alpha^2(x))L(\alpha(y))L(z)\\
&\relphantom{} + L(\alpha(zy))L(\alpha(x))\alpha - L(\alpha^2(z))L(\alpha(y))L(x)\\
&\relphantom{} + L(\alpha(zx))L(\alpha(y))\alpha - L(\alpha(y)(xz))\alpha^2.
\end{split}
\end{equation}
Note that the sum of the three terms on the right-hand side of \eqref{star'} with a plus sign is symmetric in $x$ and $y$.  Therefore, switching $x$ and $y$ in \eqref{star'}, subtracting the result from \eqref{star'}, and then using \eqref{lxyz}, we obtain the desired identity \eqref{lxyz2}.
\end{proof}

Recall from \eqref{lxy} that $L_{xy}(z) = \{xyz\}$, the triple product in \eqref{hjtriple}.  Therefore, the Hom-Jordan triple identity \eqref{homjtsid} for $A_T = (A,\{,,\},\alpha^2)$ in Theorem \ref{thm:hjhjts} becomes
\begin{equation}
\label{homjts1}
L_{\alpha^2(x)\alpha^2(y)}L_{uv} - L_{\alpha^2(u)\alpha^2(v)}L_{xy} - L_{\{xyu\}\alpha^2(v)}\alpha^2 + L_{\alpha^2(u)\{yxv\}}\alpha^2 = 0
\end{equation}
for all $u,v,x,y \in A$.  In the next several Lemmas, we compute the four terms in \eqref{homjts1} in terms of the other operators in Definition \ref{def:L}.  The following Lemma is about the first term in \eqref{homjts1}.

\begin{lemma}
\label{lem4:hj}
Let $(A,\mu,\alpha)$ be a multiplicative Hom-Jordan algebra.  Then
\begin{equation}
\label{A}
\begin{split}
L_{\alpha^2(x)\alpha^2(y)}L_{uv}
&= L(\alpha^2(xy))L(\alpha(uv))\alpha^2 + L(\alpha^2(x),\alpha^2(y))L(uv)\alpha\\
&\relphantom{} + L(\alpha^2(xy))L(\alpha(u),\alpha(v))\alpha + L(\alpha^2(x),\alpha^2(y))L(u,v)
\end{split}
\end{equation}
for all $u,v,x,y \in A$.
\end{lemma}

\begin{proof}
From Definition \ref{def:L} we have
\[
L_{\alpha^2(x)\alpha^2(y)}L_{uv} = \left(L(\alpha^2(x)\alpha^2(y))\alpha + L(\alpha^2(x),\alpha^2(y))\right)\left(L(uv)\alpha + L(u,v)\right).
\]
The desired identity is obtained by expanding this product, using the multiplicativity of $\alpha$ and Lemma \ref{lem1:hj}.
\end{proof}

Switching $(x,y)$ with $(u,v)$ in Lemma \ref{lem4:hj}, we obtain the following result, which computes the second term in \eqref{homjts1}.

\begin{lemma}
\label{lem5:hj}
Let $(A,\mu,\alpha)$ be a multiplicative Hom-Jordan algebra.  Then
\begin{equation}
\label{B}
\begin{split}
L_{\alpha^2(u)\alpha^2(v)}L_{xy}
&= L(\alpha^2(uv))L(\alpha(xy))\alpha^2 + L(\alpha^2(u),\alpha^2(v))L(xy)\alpha\\
&\relphantom{} + L(\alpha^2(uv))L(\alpha(x),\alpha(y))\alpha + L(\alpha^2(u),\alpha^2(v))L(x,y)
\end{split}
\end{equation}
for all $u,v,x,y \in A$.
\end{lemma}

Using the previous two Lemmas, the first two terms in \eqref{homjts1} are computed in the following Lemma.

\begin{lemma}
\label{lem6:hj}
Let $(A,\mu,\alpha)$ be a multiplicative Hom-Jordan algebra.  Then
\[
\begin{split}
L(\alpha^2(xy))L(\alpha(uv))\alpha^2 &- L(\alpha^2(uv))L(\alpha(xy))\alpha^2
= L(\alpha(xy),\alpha(u)\alpha(v))\alpha^2,\\
L(\alpha^2(x),\alpha^2(y))L(uv)\alpha &- L(\alpha^2(uv))L(\alpha(x),\alpha(y))\alpha = L(L(\alpha(x),\alpha(y))(uv))\alpha^3,\\
L(\alpha^2(xy))L(\alpha(u),\alpha(v))\alpha &- L(\alpha^2(u),\alpha^2(v))L(xy)\alpha\\
&= L(L(\alpha^2(v))L(\alpha(u))(xy))\alpha^3 - L(L(\alpha^2(u))L(\alpha(v))(xy))\alpha^3,\\
L(\alpha^2(x),\alpha^2(y))L(u,v) &- L(\alpha^2(u),\alpha^2(v))L(x,y)\\
&= L(x,y,u,v) + L(y,x,v,u)
\end{split}
\]
for all $u,v,x,y \in A$.
\end{lemma}

\begin{proof}
The first equality is immediate from the definition of $L(x,y)$.  The second equality holds because both sides are equal to $L(\alpha(x),\alpha(y),uv)\alpha$ by Lemma \ref{lem3:hj}.  The third equality holds because, by Lemma \ref{lem3:hj}, both sides are equal to
\[
-L(\alpha(u),\alpha(v),xy)\alpha = -L(L(\alpha(u),\alpha(v))(xy))\alpha^3.
\]
For the last equality, observe that by definition,
\[
\begin{split}
L(x,y,u,v) &= L(\alpha^2(x),\alpha^2(y))L(\alpha(u))L(v) - L(\alpha^3(u))L(\alpha(x),\alpha(y))L(v)\\
&\relphantom{} - L(\alpha^3(v))L(\alpha(x),\alpha(y))L(u) + L(\alpha^3(v))L(\alpha^2(u))L(x,y).
\end{split}
\]
Therefore, switching $(x,y)$ with $(y,x)$ and $(u,v)$ with $(v,u)$ and adding the result to $L(x,y,u,v)$, we obtain
\[
\begin{split}
L(x,y,u,v) + L(y,x,v,u)
&= L(\alpha^2(x),\alpha^2(y))\{L(\alpha(u))L(v) - L(\alpha(v))L(u)\}\\
&\relphantom{} - \{L(\alpha^3(u))L(\alpha^2(v)) - L(\alpha^3(v))L(\alpha^2(u))\}L(x,y)\\
&= L(\alpha^2(x),\alpha^2(y))L(u,v) - L(\alpha^2(u),\alpha^2(v))L(x,y).
\end{split}
\]
This proves the last equality.
\end{proof}

Next we compute the third term in \eqref{homjts1}.

\begin{lemma}
\label{lem7:hj}
Let $(A,\mu,\alpha)$ be a multiplicative Hom-Jordan algebra.  Then
\[
\begin{split}
L_{\{xyu\}\alpha^2(v)}\alpha^2
&= L(L(\alpha^2(v))L(\alpha(u))(xy))\alpha^3\\
&\relphantom{} + L(L(\alpha(x),\alpha(y))(uv))\alpha^3 - L(L(\alpha^2(u))L(x,y)(v))\alpha^3\\
&\relphantom{} - L(\alpha^2(v),(xy)\alpha(u))\alpha^2 + L(x,y,u,v)
\end{split}
\]
for all $u,v,x,y \in A$, where $\{,,\}$ is the triple product in \eqref{hjtriple}.
\end{lemma}

\begin{proof}
By definition and Lemmas \ref{lem1:hj} and \ref{lem3:hj}, we have
\begin{equation}
\label{C'}
\begin{split}
L_{\{xyu\}\alpha^2(v)}\alpha^2
&= L(\{xyu\}\alpha^2(v))\alpha^3 + L(\{xyu\},\alpha^2(v))\alpha^2\\
&= L(L(L(xy)(\alpha(u)))(\alpha^2(v)))\alpha^3 + L(L(L(x,y)(u))(\alpha^2(v)))\alpha^3\\
&\relphantom{} + L(\alpha\{xyu\})L(\alpha^2(v))\alpha^2 - L(\alpha^3(v))L(\{xyu\})\alpha^2\\
&= L(L(\alpha^2(v))L(\alpha(u))(xy))\alpha^3 + L(L(x,y,u)(v))\alpha^3\\
&\relphantom{} + L(\alpha(xy)\alpha^2(u))L(\alpha^2(v))\alpha^2 + L(\alpha(x),\alpha(y),\alpha(u))L(v)\\
&\relphantom{} - L(\alpha^3(v))L((xy)\alpha(u))\alpha^2 - L(\alpha^3(v))L(x,y,u).
\end{split}
\end{equation}
In the last equality in \eqref{C'}, we used the fact
\[
L(\{xyu\})
= L(L(xy)(\alpha(u))) + L(L(x,y)(u)).
\]
Together with Lemma \ref{lem3:hj}, this implies
\[
- L(\alpha^3(v))L(\{xyu\})\alpha^2
= - L(\alpha^3(v))L((xy)\alpha(u))\alpha^2  - L(\alpha^3(v))L(x,y,u)
\]
and
\[
\begin{split}
L(\alpha\{xyu\})L(\alpha^2(v))\alpha^2
&= L(\{\alpha(x)\alpha(y)\alpha(u)\})\alpha^2 L(v)\\
&= L(L(\alpha(x)\alpha(y))(\alpha^2(u)))\alpha^2 L(v) + L(L(\alpha(x),\alpha(y))(\alpha(u)))\alpha^2 L(v)\\
&= L(\alpha(xy)\alpha^2(u))L(\alpha^2(v))\alpha^2 +  L(\alpha(x),\alpha(y),\alpha(u))L(v).
\end{split}
\]
The Lemma follows from \eqref{C'} by Definition \ref{def:L}.
\end{proof}

With an argument similar to the proof of Lemma \ref{lem7:hj}, we compute the fourth term in \eqref{homjts1}.

\begin{lemma}
\label{lem8:hj}
Let $(A,\mu,\alpha)$ be a multiplicative Hom-Jordan algebra.  Then
\[
\begin{split}
L_{\alpha^2(u)\{yxv\}}\alpha^2
&= L(L(\alpha^2(u))L(\alpha(v))(xy))\alpha^3 - L(L(\alpha^2(u))L(x,y)(v))\alpha^3\\
&\relphantom{} + L(\alpha^2(u),\alpha(v)(xy))\alpha^2 - L(y,x,v,u)
\end{split}
\]
for all $u,v,x,y \in A$, where $\{,,\}$ is the triple product in \eqref{hjtriple}.
\end{lemma}

\begin{proof}
As in the proof of Lemma \ref{lem7:hj}, we compute as follows:
\[
\begin{split}
L_{\alpha^2(u)\{yxv\}}\alpha^2
&= L(\alpha^2(u)\{yxv\})\alpha^3 - L(\{yxv\},\alpha^2(u))\alpha^2\\
&= L(L(\alpha^2(u))L(yx)(\alpha(v)))\alpha^3 + L(L(\alpha^2(u))L(y,x)(v))\alpha^3\\
&\relphantom{} + L(\alpha^3(u))L((xy)\alpha(v))\alpha^2 - L(\alpha(xy)\alpha^2(v))L(\alpha^2(u))\alpha^2\\
&\relphantom{} + L(\alpha^3(u))L(y,x,v) - L(\alpha(y),\alpha(x),\alpha(v))L(u).
\end{split}
\]
The desired equality now follows from the commutativity of $\mu$, Definition \ref{def:L}, and Lemma \ref{lem1:hj}.
\end{proof}

We are now ready to prove Theorem \ref{thm:hjhjts}.

\begin{proof}[Proof of Theorem \ref{thm:hjhjts}]
It is clear that $A_T = (A,\{,,\},\alpha^2)$ is multiplicative and that the triple product $\{,,\}$ is outer-symmetric, since $\mu$ is commutative.  It remains to prove the Hom-Jordan triple identity \eqref{homjtsid} for $A_T$, which takes the form \eqref{homjts1}.  Using Lemma \ref{lem4:hj} - Lemma \ref{lem8:hj}, the left-hand side of \eqref{homjts1} is equal to
\[
\begin{split}
L(\alpha(xy),\alpha(u)\alpha(v))\alpha^2 &+ L(\alpha^2(v),(xy)\alpha(u))\alpha^2 + L(\alpha^2(u),\alpha(v)(xy))\alpha^2\\
&= \cyclicsum_{(xy,\alpha(u),\alpha(v))}\, L(\alpha(xy),\alpha(u)\alpha(v))\alpha^2,
\end{split}
\]
in which the cyclic sum is taken over $xy$, $\alpha(u)$, and $\alpha(v)$.  This cyclic sum is equal to $0$ by Lemma \ref{lem2:hj}, as desired.
\end{proof}

If $\alpha = Id$ in Theorem \ref{thm:hjhjts}, then we obtain the following well-known method of obtaining a Jordan triple system from a Jordan algebra.

\begin{corollary}
\label{cor1:hj}
Let $A$ be a Jordan algebra.  Then $(A,\{,,\})$ is a Jordan triple system, where
\[
\{xyz\} = x(yz) + (xy)z - y(xz)
\]
for $x,y,z \in A$.
\end{corollary}

%%%%%%%%%%%%%%%%%%%%%%%%%%%
\section{Hom-Nambu algebras from Hom-Jordan and Hom-alternative algebras}
\label{sec:homn}
%%%%%%%%%%%%%%%%%%%%%%%%%%%

In this section, we discuss how Hom-Nambu algebras (Definition \ref{def:homnambulie}) of different arities arise from Hom-Jordan and Hom-alternative algebras.

Recall from Theorem \ref{thm:jtslts} that every Hom-Jordan triple system with equal twisting maps has an underlying Hom-Lie triple system. Therefore, combining Theorems \ref{thm:jtslts} and \ref{thm:hjhjts}, we obtain the following method of constructing a Hom-Lie triple system, and hence a ternary Hom-Nambu algebra, from a Hom-Jordan algebra.

\begin{corollary}
\label{cor2:hj}
Let $(A,\mu,\alpha)$ be a multiplicative Hom-Jordan algebra.   Then
\[
A_L = (A,[,,],\alpha^2)
\]
is a multiplicative Hom-Lie triple system, where
\begin{equation}
\label{albracket}
[xyz] = 2\left(\alpha(x)(yz) - \alpha(y)(xz)\right)
\end{equation}
for $x,y,z \in A$.
\end{corollary}

\begin{proof}
The expression for the triple product $[,,]$ is obtained from \eqref{hjtriple} by switching $x$ and $y$ and subtracting the result from \eqref{hjtriple}.
\end{proof}

It is shown in \cite{yau13} that every multiplicative $n$-ary Hom-Nambu algebra gives rise to a sequence of Hom-Nambu algebras of exponentially higher arities.  More precisely, we recall the following result.  We use the abbreviations in \eqref{xij}.

\begin{theorem}[\cite{yau13}]
\label{thm:higher}
Let $(L,\bracket,\alpha)$ be a multiplicative $n$-ary Hom-Nambu algebra.  For $k \geq 0$ define the $(2^k(n-1)+1)$-ary product $\bracket^{(k)}$ inductively by setting
\[
\begin{split}
\bracket^{(0)} &= \bracket,\\
[x_{1,2^k(n-1)+1}]^{(k)} &= [[x_{1,2^{k-1}(n-1)+1}]^{(k-1)}, \alpha^{2^{k-1}}(x_{2^{k-1}(n-1)+2,2^k(n-1)+1})]^{(k-1)}
\end{split}
\]
for $k \geq 1$ and $x_i \in L$.  Then
\[
L^k = (L,\bracket^{(k)},\alpha^{2^k})
\]
is a multiplicative $(2^k(n-1)+1)$-ary Hom-Nambu algebra for each $k \geq 0$.
\end{theorem}

In particular, when $L$ is ternary, $L^k$ is $(2^{k+1}+1)$-ary.  Therefore, combining Corollary \ref{cor2:hj} and Theorem \ref{thm:higher}, we obtain a sequence of Hom-Nambu algebras of different arities from a multiplicative Hom-Jordan algebra.

\begin{corollary}
\label{cor1:higher}
Let $(A,\mu,\alpha)$ be a multiplicative Hom-Jordan algebra.  For $k \geq 0$ define the $(2^{k+1}+1)$-ary product $\bracket^{(k)}$ inductively by setting
\[
[x_1,x_2,x_3]^{(0)} = 2\left(\alpha(x_1)(x_2x_3) - \alpha(x_2)(x_1x_3)\right)
\]
as in \eqref{albracket} and
\[
[x_{1,2^{k+1}+1}]^{(k)} = [[x_{1,2^k+1}]^{(k-1)},\alpha^{2^k}(x_{2^k+2,2^{k+1}+1})]^{(k-1)}
\]
for $k \geq 1$ and $x_i \in A$.  Then
\[
A_L^k = (A,\bracket^{(k)},\alpha^{2^{k+1}})
\]
is a multiplicative $(2^{k+1}+1)$-ary Hom-Nambu algebra for each $k \geq 0$.
\end{corollary}

We have $A_L^0 = A_L = (A,[,,],\alpha^2)$ as in Corollary \ref{cor2:hj}.  Let us discuss $A_L^1$ and $A_L^2$ in the following example.

%%%%%%%%%%%%%%%%%%
\begin{example}
\label{ex:jhigher}
Let $(A,\mu,\alpha)$ be a multiplicative Hom-Jordan algebra.  By Corollary \ref{cor1:higher} there is a multiplicative $5$-ary Hom-Nambu algebra
\[
A_L^1 = (A,\bracket^{(1)},\alpha^4)
\]
with
\[
[x_1,\ldots,x_5]^{(1)} = [[x_1,x_2,x_3],\alpha^2(x_4),\alpha^2(x_5)].
\]
Using multiplicativity and \eqref{albracket}
\[
[xyz] = 2\left(\alpha(x)(yz) - \alpha(y)(xz)\right),
\]
we can express $\bracket^{(1)}$ in terms of $\mu$ and $\alpha$ as:
\[
\begin{split}
\frac{1}{4}[x_{1,5}]^{(1)}
&= \alpha(\alpha(x_1)(x_2x_3))(\alpha^2(x_4x_5)) - \alpha^3(x_4)\left((\alpha(x_1)(x_2x_3))\alpha^2(x_5)\right)\\
&\relphantom{} - \alpha(\alpha(x_2)(x_1x_3))(\alpha^2(x_4x_5)) + \alpha^3(x_4)\left((\alpha(x_2)(x_1x_3))\alpha^2(x_5)\right).
\end{split}
\]
Likewise, by Corollary \ref{cor1:higher} there is a multiplicative $9$-ary Hom-Nambu algebra
\[
A_L^2 = (A,\bracket^{(2)},\alpha^8)
\]
with
\[
[x_{1,9}]^{(2)} = [[x_{1,5}]^{(1)},\alpha^4(x_{6,9})]^{(1)}.
\]
One can expand $\bracket^{(2)}$ in terms of $\mu$ and $\alpha$ as above.  There are sixteen terms in the expanded expression.
\qed
\end{example}
%%%%%%%%%%%%%%%%%%

Recall from Theorem \ref{thm:homalt} that the plus Hom-algebra of a multiplicative Hom-alternative algebra is a multiplicative Hom-Jordan algebra.  Therefore, combining Theorems \ref{thm:homalt} and \ref{thm:hjhjts}, we obtain the following method of constructing a Hom-Jordan triple system from a Hom-alternative algebra.

\begin{corollary}
\label{cor3:hj}
Let $(A,\mu,\alpha)$ be a multiplicative Hom-alternative algebra.  Then
\[
A^+_T = (A,\{,,\},\alpha^2)
\]
is a multiplicative Hom-Jordan triple system, where
\begin{equation}
\label{startriple}
\{xyz\} = \frac{1}{2}\left(\alpha(x)(yz) + \alpha(z)(yx)\right)
\end{equation}
for $x,y,z \in A$.
\end{corollary}

\begin{proof}
The plus Hom-algebra $A^+ = (A,\ast,\alpha)$ is a multiplicative Hom-Jordan algebra (Theorem \ref{thm:homalt}), where $\ast = (\mu + \muop)/2$.  By Theorem \ref{thm:hjhjts} there is a multiplicative Hom-Jordan triple system $(A,\{,,\},\alpha^2)$, where
\[
\{xyz\}
= \alpha(x) \ast (y \ast z) + (x \ast y) \ast \alpha(z) - \alpha(y) \ast (x \ast z).
\]
Expanding this triple product in terms of $\mu$, we obtain:
\begin{equation}
\label{startriple'}
\begin{split}
4\{xyz\}
&= \alpha(x)(yz) + \alpha(x)(zy) + (yz)\alpha(x) + (zy)\alpha(x)\\
&\relphantom{} + (xy)\alpha(z) + (yx)\alpha(z) + \alpha(z)(xy) + \alpha(z)(yx)\\
&\relphantom{} - \alpha(y)(xz) - \alpha(y)(zx) - (xz)\alpha(y) - (zx)\alpha(y).
\end{split}
\end{equation}
Dividing the right-hand side of \eqref{startriple'} into six pairs and using the anti-symmetry of the Hom-associator \eqref{homassociator}, we have:
\begin{equation}
\label{startriple''}
\begin{split}
\alpha(x)(zy) - (xz)\alpha(y) &= -as_A(x,z,y) = as_A(x,y,z),\\
(yz)\alpha(x) - \alpha(y)(zx) &= as_A(y,z,x) = as_A(x,y,z),\\
(yx)\alpha(z) - \alpha(y)(xz) &= as_A(y,x,z) = -as_A(x,y,z),\\
\alpha(z)(xy) - (zx)\alpha(y) &= -as_A(z,x,y) = -as_A(x,y,z),\\
\alpha(x)(yz) + (xy)\alpha(z) &= as_A(x,y,z) + 2\alpha(x)(yz),\\
(zy)\alpha(x) + \alpha(z)(yx) &= as_A(z,y,x) + 2\alpha(z)(yx)\\
&= -as_A(x,y,z) + 2\alpha(z)(yx).
\end{split}
\end{equation}
The desired equality \eqref{startriple} is obtained by using \eqref{startriple''} in \eqref{startriple'}.
\end{proof}

Combining Theorem \ref{thm:jtslts} with Corollary \ref{cor3:hj}, we obtain the following method of constructing a Hom-Lie triple system, and hence a ternary Hom-Nambu algebra, from a multiplicative Hom-alternative algebra.

\begin{corollary}
\label{cor4:hj}
Let $(A,\mu,\alpha)$ be a multiplicative Hom-alternative algebra.  Then
\[
A^+_L = (A,[,,],\alpha^2)
\]
is a multiplicative Hom-Lie triple system, where
\begin{equation}
\label{starlproduct}
[xyz] = \frac{1}{2}[[x,y],\alpha(z)] - as_A(x,y,z)
\end{equation}
with $[,] = \mu - \muop$ the commutator bracket of $\mu$ and $as_A$ the Hom-associator \eqref{homassociator}.
\end{corollary}

\begin{proof}
The multiplicative Hom-Jordan triple system $A^+_T = (A,\{,,\},\alpha^2)$ in Corollary \ref{cor3:hj} yields a multiplicative Hom-Lie triple system $(A,[,,],\alpha^2)$ (Theorem \ref{thm:jtslts}) with
\[
[xyz] = \{xyz\} - \{yxz\}.
\]
We must show that this triple product is equal to the one in \eqref{starlproduct}.  To prove this, switching $x$ and $y$ in \eqref{startriple}, subtracting the result from \eqref{startriple}, and using the anti-symmetry of the Hom-associator, we obtain:
\[
\begin{split}
2[xyz]
&= \alpha(x)(yz) - \alpha(y)(xz) + \alpha(z)(yx) - \alpha(z)(xy)\\
&= -as_A(x,y,z) + (xy)\alpha(z) + as_A(y,x,z) - (yx)\alpha(z)\\
&\relphantom{} + \alpha(z)(yx) - \alpha(z)(xy)\\
&= -2as_A(x,y,z) + [[x,y],\alpha(z)].
\end{split}
\]
The desired equality \eqref{starlproduct} now follows immediately.
\end{proof}

Since Hom-Lie triple systems are also ternary Hom-Nambu algebras, combining Theorem \ref{thm:higher} and Corollary \ref{cor4:hj}, we obtain a sequence of Hom-Nambu algebras of different arities from a multiplicative Hom-alternative algebra.

\begin{corollary}
\label{cor2:higher}
Let $(A,\mu,\alpha)$ be a multiplicative Hom-alternative algebra.  For $k \geq 0$ define the $(2^{k+1}+1)$-ary product $\bracket^{(k)}$ inductively by setting
\[
[x_1,x_2,x_3]^{(0)} =\frac{1}{2}[[x_1,x_2],\alpha(x_3)] - as_A(x_1,x_2,x_3)
\]
as in \eqref{starlproduct} and
\[
[x_{1,2^{k+1}+1}]^{(k)} = [[x_{1,2^k+1}]^{(k-1)},\alpha^{2^k}(x_{2^k+2,2^{k+1}+1})]^{(k-1)}
\]
for $k \geq 1$ and $x_i \in A$.  Then
\[
A_L^{+k} = (A,\bracket^{(k)},\alpha^{2^{k+1}})
\]
is a multiplicative $(2^{k+1}+1)$-ary Hom-Nambu algebra for each $k \geq 0$.
\end{corollary}

We have $A_L^{+0} = A_L^+ = (A,[,,],\alpha^2)$ as in Corollary \ref{cor4:hj}.  Let us discuss $A_L^{+1}$ and $A_L^{+2}$ in the following example.

%%%%%%%%%%%%%%%%%%
\begin{example}
\label{ex2:jhigher}
Let $(A,\mu,\alpha)$ be a multiplicative Hom-alternative algebra.  By Corollary \ref{cor2:higher} there is a multiplicative $5$-ary Hom-Nambu algebra
\[
A_L^{+1} = (A,\bracket^{(1)},\alpha^4).
\]
Recall from \eqref{starlproduct} that
\[
[xyz] = \frac{1}{2}[[x,y],\alpha(z)] - as_A(x,y,z),
\]
where $[,] = \mu - \muop$ is the commutator bracket of $\mu$ and $as_A$ is the Hom-associator \eqref{homassociator}.  Therefore, we can express the $5$-ary product as
\[
\begin{split}
[x_{1,5}]^{(1)}
&= [[x_1,x_2,x_3],\alpha^2(x_4),\alpha^2(x_5)]\\
&= \frac{1}{4}[[[[x_1,x_2],\alpha(x_3)],\alpha^2(x_4)], \alpha^3(x_5)]\\
&\relphantom{} - \frac{1}{2}as_A([[x_1,x_2],\alpha(x_3)], \alpha^2(x_4),\alpha^2(x_5)) - \frac{1}{2}[[as_A(x_1,x_2,x_3), \alpha^2(x_4)], \alpha^3(x_5)]\\
&\relphantom{} + as_A(as_A(x_1,x_2,x_3), \alpha^2(x_4), \alpha^2(x_5)).
\end{split}
\]
Likewise, by Corollary \ref{cor2:higher} there is a multiplicative $9$-ary Hom-Nambu algebra
\[
A_L^{+2} = (A,\bracket^{(2)},\alpha^8)
\]
with
\[
[x_{1,9}]^{(2)} = [[x_{1,5}]^{(1)},\alpha^4(x_{6,9})]^{(1)}.
\]
One can expand $\bracket^{(2)}$ in terms of the commutator bracket, the Hom-associator, and $\alpha$ as above.  There are sixteen terms in the expanded expression.
\qed
\end{example}
%%%%%%%%%%%%%%%%%%

%%%%%%%%%%%%%%%%%%%%%%%%%%%%%%%%%%%%%%
\section{$n$-ary Hom-Maltsev algebras}
\label{sec:maltsev}
%%%%%%%%%%%%%%%%%%%%%%%%%%%%%%%%%%%%%%

% def: Hom-Maltsev algebra; other characterizations; Hom-Maltsev admissibility of Hom-alternative algebras

% def: n-ary Hom-Maltsev algebras
% special cases: Pozhidaev's n-ary Maltsev and Hom-Maltsev algebras

In this section we introduce $n$-ary Hom-Maltsev algebras, which simultaneously generalize $n$-ary Maltsev algebras \cite{poz} and Hom-Maltsev algebras \cite{yau12}.  It is shown that the category of $n$-ary Hom-Maltsev algebras is closed under twisting by self-morphisms (Theorem \ref{thm1:hommaltsev}).  Using a special case of this result (Corollary \ref{cor2:hommaltsev}), ternary Hom-Maltsev algebras are then constructed from composition algebras (Corollary \ref{cor3:hommaltsev}).  Under some mild conditions, these ternary Hom-Maltsev algebras are not ternary Hom-Nambu-Lie algebras.  As an example, we show that the octonion algebra carries many different ternary Hom-Maltsev algebra structures (Example \ref{ex:octonion}).  The relationship between $n$-ary Hom-Maltsev algebras and $n$-ary Hom-Nambu-Lie algebras will be discussed in the next section.

To motivate the definition below, recall that a \textbf{Maltsev algebra} \cite{maltsev} $(A,\mu)$ has an anti-symmetric binary multiplication $\mu$ that satisfies the \textbf{Maltsev identity}
\[
J'(x,y,xz) = J'(x,y,z)x
\]
for $x,y,z \in A$, where
\[
J'(x,y,z) = (xy)z + (zx)y + (yz)x
\]
is the usual Jacobian.  Using anti-symmetry the Maltsev identity can be rewritten in operator form as
\begin{equation}
\label{maltsevid}
M(x,y) = R_xR_{xy} + R_{xy}R_x - R_x^2R_y + R_yR_x^2 = 0,
\end{equation}
where $R_x$ is the right multiplication operator acting from the right, i.e., $aR_x = ax$.

The following definition is the $n$-ary Hom-version of the operator $M(x,y)$.  We use the abbreviations in \eqref{xij}.

\begin{definition}
\label{def:hommaltsevian}
Let $(A,\bracket,\alpha=(\alpha_1,\ldots,\alpha_{n-1}))$ be an $n$-ary Hom-algebra with $n \geq 2$.  The \textbf{$n$-ary Hom-Maltsevian} is the $(2n-1)$-linear map $M^n_A \colon A^{\otimes 2n-1} \to A$ defined as
\begin{equation}
\label{hommaltsevian}
\begin{split}
M^n_A(z;x_{2,n};y_{2,n})
&= \sum_{i=2}^n \alpha_1\left\{[[z,x_{2,n}],\alpha_{2,i-1}(x_{2,i-1}),[x_i,y_{2,n}], \alpha_{i,n-1}(x_{i+1,n})]\right\}\\
&\relphantom{} + \sum_{i=2}^n \,[[\alpha_1(z),\alpha_{2,i-1}(x_{2,i-1}), [x_i,y_{2,n}],\alpha_{i,n-1}(x_{i+1,n})], \alpha^2_{1,n-1}(x_{2,n})]\\
&\relphantom{} - [[[z,x_{2,n}],\alpha_{1,n-1}(x_{2,n})], \alpha^2_{1,n-1}(y_{2,n})]\\
&\relphantom{} + [[[z,y_{2,n}],\alpha_{1,n-1}(x_{2,n})], \alpha^2_{1,n-1}(x_{2,n})]
\end{split}
\end{equation}
for $z,x_j,y_k \in A$.
\end{definition}

\begin{definition}
\label{def:hommaltsev}
An \textbf{$n$-ary Hom-Maltsev algebra} is an $n$-ary Hom-algebra $A$ whose $n$-ary product is anti-symmetric and that satisfies the \textbf{$n$-ary Hom-Maltsev identity} $M^n_A = 0$.
\end{definition}

An $n$-ary Hom-Maltsev algebra in which all the twisting maps are equal to the identity map is exactly an \textbf{$n$-ary Maltsev algebra} as defined by Pozhidaev \cite{poz}.  In this case, we call $M^n_A$ the \textbf{$n$-ary Maltsevian} and $M^n_A=0$ the \textbf{$n$-ary Maltsev identity}.   If, in addition, $n=2$ then we recover the definition of a Maltsev algebra \cite{maltsev}.  It is shown in \cite{poz} (Lemma 1.1) that every $n$-ary Nambu-Lie algebra (called an $n$-Lie algebra there) is also an $n$-ary Maltsev algebra.  The Hom-version of this fact is proved in the next section.

In the case $n=2$, using anti-symmetry the binary Hom-Maltsev identity $M^2_A = 0$ is equivalent to
\[
\alpha([[x,y],[x,z]])
= [[[x,y],\alpha(z)],\alpha^2(x)]
+ [[[z,x],\alpha(x)],\alpha^2(y)]
+ [[[y,z],\alpha(x)],\alpha^2(x)],
\]
which is equivalent to the Hom-Maltsev identity in \cite{yau12}.  A \textbf{Hom-Maltsev algebra} as defined in \cite{yau12} is exactly a multiplicative binary Hom-Maltsev algebra as in Definition \ref{def:hommaltsev}.

Our first objective is to show that the category of $n$-ary Hom-Maltsev algebras is closed under twisting by self-morphisms.  We need the following observations.

\begin{lemma}
\label{lem1:hommaltsev}
Let $(A,\bracket,\alpha=(\alpha_1,\ldots,\alpha_{n-1}))$ be an $n$-ary Hom-algebra and $\beta \colon A \to A$ be a morphism.  Define $\bracket_\beta = \beta \bracket$, $\beta\alpha = (\beta\alpha_1,\ldots,\beta\alpha_{n-1})$, and the $n$-ary Hom-algebra
\[
A_\beta = (A,\bracket_\beta,\beta\alpha).
\]
Then the following statements hold.
\begin{enumerate}
\item
If $A$ is multiplicative, then so is $A_\beta$.
\item
If $\bracket$ is anti-symmetric, then so is $\bracket_\beta$.
\item
$M^n_{A_\beta} = \beta^3 M^n_A$.
\end{enumerate}
\end{lemma}

\begin{proof}
All three statements are immediate from the definitions.
\end{proof}

% twisting by self-morphisms
The following result is an immediate consequence of Lemma \ref{lem1:hommaltsev}.

\begin{theorem}
\label{thm1:hommaltsev}
Let $(A,\bracket,\alpha=(\alpha_1,\ldots,\alpha_{n-1}))$ be an $n$-ary Hom-Maltsev algebra and $\beta \colon A \to A$ be a morphism.  Then $A_\beta = (A,\bracket_\beta,\beta\alpha)$, in which
\[
\bracket_\beta = \beta \bracket \quad\text{and}\quad
\beta\alpha = (\beta\alpha_1,\ldots,\beta\alpha_{n-1}),
\]
is also an $n$-ary Hom-Maltsev algebra.  Moreover, if $A$ is multiplicative, then so is $A_\beta$.
\end{theorem}

In other words, Theorem \ref{thm1:hommaltsev} says that the category of $n$-ary Hom-Maltsev algebras is closed under twisting by self-morphisms.  The corresponding closure property for $n$-ary Hom-Nambu(-Lie) algebras can be found in \cite{yau13}.  Other Hom-type algebras can also be shown to have the same kind of closure property.   We emphasize that this closure property is unique to Hom-type algebras, since ordinary algebras are usually not closed under twisting by morphisms.  In fact, under twisting by self-morphisms, an ordinary algebra should give rise to a Hom-type algebra, as in Corollary \ref{cor2:hommaltsev} below.  Such a twisting result for Hom-type algebras was first proved by the author in \cite{yau2}.

Let us discuss some special cases of Theorem \ref{thm1:hommaltsev}.

\begin{corollary}
\label{cor1:hommaltsev}
Let $(A,\bracket,\alpha)$ be a multiplicative $n$-ary Hom-Maltsev algebra.  Then
\[
A_\alpha = (A,\bracket_\alpha=\alpha\bracket,\alpha^2)
\]
is also a multiplicative $n$-ary Hom-Maltsev algebra.
\end{corollary}

\begin{proof}
This is the special case of Theorem \ref{thm1:hommaltsev} with $\beta = \alpha$.
\end{proof}

Using Corollary \ref{cor1:hommaltsev} repeatedly, we obtain the following result.

\begin{corollary}
\label{cor1.5:hommaltsev}
Let $(A,\bracket,\alpha)$ be a multiplicative $n$-ary Hom-Maltsev algebra.  Then
\[
A_k = (A,\bracket_k = \alpha^{2^k-1}\bracket,\alpha^{2^k})
\]
is also a multiplicative $n$-ary Hom-Maltsev algebra for each $k \geq 0$.
\end{corollary}

\begin{corollary}
\label{cor2:hommaltsev}
Let $(A,\bracket)$ be an $n$-ary Maltsev algebra and $\beta \colon A \to A$ be a morphism.  Then
\[
A_\beta = (A,\bracket_\beta = \beta\bracket,\beta)
\]
is a multiplicative $n$-ary Hom-Maltsev algebra.
\end{corollary}

\begin{proof}
This is the special case of Theorem \ref{thm1:hommaltsev} with $\alpha_i = Id$ for all $i$.
\end{proof}

Using Corollary \ref{cor2:hommaltsev} one can construct $n$-ary Hom-Maltsev algebras from $n$-ary Maltsev algebras.  We now discuss how ternary Hom-Maltsev algebras arise from composition algebras.

Let us first recall the definition of a composition algebra, which comes up in connection with the Cayley-Dickson double construction.  For example, the octonion algebra can be constructed as a double of the quaternion algebra, which in turn can be constructed as a double of the complex numbers.  The reader may consult, e.g., \cite{rowen} (Appendix 21B) for discussion of composition algebras.

\begin{definition}
Let $(A,\mu)$ be a binary unital algebra with multiplicative identity $1$.
\begin{enumerate}
\item
An \textbf{involution} $\ast \colon A \to A$ on $A$ is an anti-automorphism of order $2$, i.e., $\ast \not= Id$ is a linear automorphism such that $(x^*)^* = x$ and $(xy)^* = y^*x^*$ for all $x,y \in A$.
\item
We say that $A$ is a \textbf{composition algebra with involution $\ast$} if $\ast$ is an involution on $A$ such that 
\begin{enumerate}
\item
$xx^* = x^*x \in \bk 1$ for all $x \in A$, and
\item
the symmetric bilinear form on $A$
\begin{equation}
\label{bform}
\langle x,y\rangle = \frac{1}{2}(xy^* + yx^*) \in \bk 1
\end{equation}
is non-degenerate.
\end{enumerate}
\end{enumerate}
\end{definition}

If $Q(x)$ denotes the quadratic form $Q(x) = xx^* \in \bk 1$, then the symmetric bilinear form \eqref{bform} becomes
\[
\langle x,y\rangle = \frac{1}{2}(Q(x+y) - Q(x) - Q(y)).
\]
The quadratic form $Q$ is usually called the \emph{norm} on $A$.  The following consequence of Corollary \ref{cor2:hommaltsev} tells us how to construct ternary Hom-Maltsev algebras from composition algebras. It extends some results due to Pozhidaev \cite{poz}.

% Corollary: ternary Hom-Maltsev from composition algebras
\begin{corollary}
\label{cor3:hommaltsev}
Let $(A,\ast)$ be a composition algebra with involution $\ast$ and $\alpha \colon A \to A$ be an algebra morphism that commutes with $\ast$ and preserves the multiplicative identity of $A$.  Then
\[
M(A)_\alpha = (A,[,,]_\alpha ,\alpha)
\]
is a multiplicative ternary Hom-Maltsev algebra, where
\[
[xyz]_\alpha = \alpha((xy^*)z) - \langle y,z\rangle \alpha(x) + \langle x,z\rangle\alpha(y) - \langle x,y\rangle \alpha(z).
\]
Moreover, if $\dim(A) \geq 5$ and $\alpha$ is injective, then $M(A)_\alpha$ is not a ternary Hom-Nambu-Lie algebra.
\end{corollary}

\begin{proof}
Pozhidaev proved in \cite{poz} (Theorem 2.3) that $M(A) = (A,[,,])$ is a ternary Maltsev algebra, where
\[
[xyz] = (xy^*)z - \langle y,z\rangle x + \langle x,z\rangle y - \langle x,y\rangle z.
\]
Moreover, Pozhidaev proved in \cite{poz} (Theorem 3.1) that, if $\dim(A) \geq 5$, then $M(A)$ is not a ternary Nambu-Lie algebra, i.e., the ternary Jacobian $J^3_{M(A)} \not= 0$.  The assumptions on $\alpha$ imply that
\[
Q(\alpha(x)) = xx^* = Q(x)
\]
for all $x \in A$, which in turn implies that
\[
\langle \alpha(x),\alpha(y)\rangle = \langle x,y\rangle
\]
for all $x,y \in A$.  It follows that $\alpha \colon M(A) \to M(A)$ is a morphism of ternary Maltsev algebras.  The first assertion now follows from the $n=3$ case of Corollary \ref{cor2:hommaltsev}.

For the second assertion, observe that the ternary Hom-Jacobian of $M(A)_\alpha = (A,[,,]_\alpha=\alpha[,,],\alpha)$ and the ternary Jacobian of $M(A)$ are related as
\[
J^3_{M(A)_\alpha} = \alpha^2 J^3_{M(A)}.
\]
If $\dim(A) \geq 5$, then $J^3_{M(A)} \not= 0$, which implies that $J^3_{M(A)_\alpha} \not= 0$ if $\alpha$ is injective.  So in this case $M(A)_\alpha$ does not satisfy the ternary Hom-Nambu identity and hence is not a ternary Hom-Nambu-Lie algebra.
\end{proof}

% example: ternary Hom-Maltsev from octonions; not Hom-Nambu-Lie
In the following example, using Corollary \ref{cor3:hommaltsev} we show that the octonion algebra carries many different ternary Hom-Maltsev algebra structures that are not ternary Hom-Nambu-Lie algebras.

%%%%%%%%%%%%%%%%%%%
\begin{example}
\label{ex:octonion}
The octonion algebra $\oct$ \cite{baez,rowen,schafer} is the eight-dimensional alternative (but not associative) algebra with basis $\{e_0,\ldots,e_7\}$ and the following multiplication table:
\begin{center}
\begin{tabular}{c|c|c|c|c|c|c|c|c}
$\mu$ & $e_0$ & $e_1$ & $e_2$ & $e_3$ & $e_4$ & $e_5$ & $e_6$ & $e_7$ \\\hline
$e_0$ & $e_0$ & $e_1$ & $e_2$ & $e_3$ & $e_4$ & $e_5$ & $e_6$ & $e_7$ \\\hline
$e_1$ & $e_1$ & $-e_0$ & $e_4$ & $e_7$ & $-e_2$ & $e_6$ & $-e_5$ & $-e_3$ \\\hline
$e_2$ & $e_2$ & $-e_4$ & $-e_0$ & $e_5$ & $e_1$ & $-e_3$ & $e_7$ & $-e_6$ \\\hline
$e_3$ & $e_3$ & $-e_7$ & $-e_5$ & $-e_0$ & $e_6$ & $e_2$ & $-e_4$ & $e_1$ \\\hline
$e_4$ & $e_4$ & $e_2$ & $-e_1$ & $-e_6$ & $-e_0$ & $e_7$ & $e_3$ & $-e_5$ \\\hline
$e_5$ & $e_5$ & $-e_6$ & $e_3$ & $-e_2$ & $-e_7$ & $-e_0$ & $e_1$ & $e_4$ \\\hline
$e_6$ & $e_6$ & $e_5$ & $-e_7$ & $e_4$ & $-e_3$ & $-e_1$ & $-e_0$ & $e_2$ \\\hline
$e_7$ & $e_7$ & $e_3$ & $e_6$ & $-e_1$ & $e_5$ & $-e_4$ & $-e_2$ & $-e_0$ \\
\end{tabular}
\end{center}
For an octonion $x = \sum_{i=0}^7 b_ie_i$ with each $b_i \in \bk$, its \textbf{conjugate} is defined as the octonion
\[
\xbar = b_0e_0 - \sum_{i=1}^7 b_ie_i.
\]
The octonion algebra $\oct$ is a composition algebra with involution given by conjugation \cite{rowen} (Remark 21B.16).

To use Corollary \ref{cor3:hommaltsev} on $\oct$, consider, for example, the algebra automorphism $\alpha \colon \oct \to \oct$ given by
\begin{equation}
\label{octaut}
\begin{split}
\alpha(e_0) = e_0,\quad \alpha(e_1) = e_5,\quad \alpha(e_2) = e_6,\quad \alpha(e_3) = e_7,\\
\alpha(e_4) = e_1,\quad \alpha(e_5) = e_2,\quad \alpha(e_6) = e_3,\quad \alpha(e_7) = e_4.
\end{split}
\end{equation}
This map $\alpha$ preserves the multiplicative identity $e_0$ and commutes with conjugation.  By Corollary \ref{cor3:hommaltsev} there is a multiplicative ternary Hom-Maltsev algebra
\[
M(\oct)_\alpha = (\oct,[,,]_\alpha,\alpha)
\]
that is not a ternary Hom-Nambu-Lie algebra.  Moreover, $(\oct,[,,]_\alpha)$ is not a ternary Maltsev algebra.  Indeed, one can check that the ternary Maltsevian of $(\oct,[,,]_\alpha)$ satisfies
\[
M^3(e_1;e_2,e_4;e_6,e_7) = -e_2 - e_5 \not= 0.
\]
Therefore, $(\oct,[,,]_\alpha)$ does not satisfy the ternary Maltsev identity.

There is a more conceptual description of the algebra automorphism $\alpha$ in \eqref{octaut}.  Note that $e_1$ and $e_2$ anti-commute, and $e_3$ anti-commutes with $e_1$, $e_2$, and $e_1e_2 = e_4$. Such a triple $(e_1,e_2,e_3)$ is called a \textbf{basic triple} in \cite{baez}.  Another basic triple is $(e_5,e_6,e_7)$.  Then $\alpha$ in \eqref{octaut} is the unique automorphism on $\oct$ that sends the basic triple $(e_1,e_2,e_3)$ to the basic triple $(e_5,e_6,e_7)$.

If we take any two basic triples involving only the basis elements $\{e_1,\ldots,e_7\}$, then there is a unique automorphism $\beta$ on $\oct$ that takes one to the other.  Such an automorphism must preserve the multiplicative identity and commute with conjugation.  Corollary \ref{cor3:hommaltsev} can then be applied to yield a multiplicative ternary Hom-Maltsev algebra $M(\oct)_\beta$ that is not a ternary Hom-Nambu-Lie algebra.
\qed
\end{example}
%%%%%%%%%%%%%%%%%%%

%%%%%%%%%%%%%%%%%%%%%%%%%%%%%%%%%%%%%%%
\section{$n$-ary Hom-Nambu-Lie and Hom-Maltsev algebras}
\label{sec:homnambu}
%%%%%%%%%%%%%%%%%%%%%%%%%%%%%%%%%%%%%%%

In this section, we establish two properties of $n$-ary Hom-Maltsev algebras.  First, we show that every multiplicative $n$-ary Hom-Nambu-Lie algebra is also a multiplicative $n$-ary Hom-Maltsev algebra (Theorem \ref{thm:liemaltsev}).  This is a Hom-type generalization of a result of Pozhidaev \cite{poz} (Lemma 1.1).  Second, we show that under some conditions an $n$-ary Hom-Maltsev algebra reduces to an $(n-1)$-ary Hom-Maltsev algebra (Theorem \ref{thm:reduction}).  This is also a Hom-type generalization of a result of Pozhidaev \cite{poz} (Lemma 1.2).  The Hom-Nambu(-Lie) analogue of this reduction result can be found in \cite{yau13}.

To show that multiplicative $n$-ary Hom-Nambu-Lie algebras are $n$-ary Hom-Maltsev algebras, we need the following observation about the relationship between the $n$-ary Hom-Jacobian \eqref{homjacobian} and the $n$-ary Hom-Maltsevian \eqref{hommaltsevian}.  Recall the abbreviations in \eqref{xij}.

\begin{lemma}
\label{lem:jm}
Let $(A,\bracket,\alpha)$ be a multiplicative $n$-ary Hom-algebra with $\bracket$ anti-symmetric. Then we have
\[
\begin{split}
M^n_A(z;x_{2,n};y_{2,n})
&= (-1)^nJ^n_A(\alpha(y_{2,n});[z,x_{2,n}],\alpha(x_{2,n}))\\
&\relphantom{} + (-1)^n[J^n_A(y_{2,n};z,x_{2,n}),\alpha^2(x_{2,n})]
\end{split}
\]
for all $z,x_2,\ldots,x_n,y_2,\ldots,y_n \in A$.
\end{lemma}

\begin{proof}
By anti-symmetry we have
\begin{equation}
\label{j}
(-1)^n J^n_A(y_{2,n};x_{1,n}) = \sum_{i=1}^n \,[\alpha(x_{1,i-1}),[x_i,y_{2,n}],\alpha(x_{i+1,n})] - [[x_{1,n}],\alpha(y_{2,n})].
\end{equation}
Write $w$ for $[z,x_{2,n}]$.  Starting from \eqref{j}, if we replace $y_{2,n}$ by $\alpha(y_{2,n})$, $x_1$ by $w$, and $x_{2,n}$ by $\alpha(x_{2,n})$ and use multiplicativity, then we obtain
\begin{equation}
\label{j1}
\begin{split}
(-1)^n J^n_A(\alpha(y_{2,n});w,\alpha(x_{2,n}))
&= \sum_{i=2}^n \alpha\left\{[w,\alpha(x_{2,i-1}),[x_i,y_{2,n}], \alpha(x_{i+1,n})]\right\}\\
&\relphantom{} - [[w,\alpha(x_{2,n})],\alpha^2(y_{2,n})] + [[w,\alpha(y_{2,n})],\alpha^2(x_{2,n})].
\end{split}
\end{equation}
On the other hand, starting from \eqref{j}, if we replace $x_1$ by $z$ and then take the product with $\alpha^2(x_{2,n})$, then we obtain
\begin{equation}
\label{j2}
\begin{split}
(-1)^n[J^n_A(y_{2,n};z,x_{2,n}),\alpha^2(x_{2,n})]
&= \sum_{i=2}^n \,[[\alpha(z),\alpha(x_{2,i-1}),[x_i,y_{2,n}],\alpha(x_{i+1,n})], \alpha^2(x_{2,n})]\\
&\relphantom{} + [[[z,y_{2,n}],\alpha(x_{2,n})],\alpha^2(x_{2,n})] - [[w,\alpha(y_{2,n})],\alpha^2(x_{2,n})].
\end{split}
\end{equation}
The desired equality now follows by adding \eqref{j1} and \eqref{j2} and comparing the result with the definition of the $n$-ary Hom-Maltsevian \eqref{hommaltsevian}.
\end{proof}

In an $n$-ary Hom-Nambu-Lie algebra, the $n$-ary Hom-Jacobian is equal to $0$.  Therefore, using Lemma \ref{lem:jm} we obtain immediately the following result.

% multiplicative n-ary Hom-Nambu-Lie are n-ary Hom-Maltsev
\begin{theorem}
\label{thm:liemaltsev}
Let $(L,\bracket,\alpha)$ be a multiplicative $n$-ary Hom-Nambu-Lie algebra.  Then $L$ is also a multiplicative $n$-ary Hom-Maltsev algebra.
\end{theorem}

Note that the converse of Theorem \ref{thm:liemaltsev} is false.  Indeed, as we have shown in Corollary \ref{cor3:hommaltsev} and Example \ref{ex:octonion}, there are multiplicative ternary Hom-Maltsev algebras that are not ternary Hom-Nambu-Lie algebras.

% Generalization of Pozhidaev reduction: n-ary to (n-1)-ary Hom-Maltsev

For the reduction result, we need the following observations.

\begin{lemma}
\label{lem:reduction}
Let $(A,\bracket,\alpha=(\alpha_1,\ldots,\alpha_{n-1}))$ be an $n$-ary Hom-algebra with $n \geq 3$.  Suppose $a \in A$ satisfies
\begin{enumerate}
\item
$\alpha_{n-1}(a) = a$ and
\item
$[a,y_{2,n-1},a] = 0$ for all $y_2, \ldots, y_{n-1} \in A$.
\end{enumerate}
Consider the $(n-1)$-ary Hom-algebra $A' = (A,\bracket',\alpha')$ with
\[
[x_{1,n-1}]' = [x_{1,n-1},a] \quad\text{and}\quad
\alpha' = (\alpha_1,\ldots,\alpha_{n-2}).
\]
Then the following statements hold.
\begin{enumerate}
\item
If $\bracket$ is anti-symmetric, then so is $\bracket'$.
\item
If $A$ is multiplicative, then so is $A'$.
\item
The Hom-Maltsevians \eqref{hommaltsevian} of $A$ and $A'$ are related as
\[
M^{n-1}_{A'}(z;x_{2,n-1};y_{2,n-1}) = M^n_A(z;x_{2,n-1},a;y_{2,n-1},a)
\]
for all $z,x_2,\ldots,x_{n-1},y_2,\ldots,y_{n-1} \in A$.
\end{enumerate}
\end{lemma}

\begin{proof}
The first two assertions are immediate from the definition of $\bracket'$ and the assumption $\alpha_{n-1}(a) = a$.  For the last assertion, substitute $x_n = y_n = a$ in the $n$-ary Hom-Maltsevian $M^n_A(z;x_{2,n};y_{2,n})$ of $A$ \eqref{hommaltsevian}.  In the first two sums in \eqref{hommaltsevian}, the $i=n$ terms both involve
\[
[x_n,y_{2,n}] = [a,y_{2,n-1},a],
\]
which is equal to $0$ by assumption.  The remaining terms give the $(n-1)$-ary Hom-Maltsevian of $A'$.
\end{proof}

Using Lemma \ref{lem:reduction} we now have the following reduction result.

\begin{theorem}
\label{thm:reduction}
Let $(A,\bracket,\alpha=(\alpha_1,\ldots,\alpha_{n-1}))$ be an $n$-ary Hom-Maltsev algebra with $n \geq 3$.  Suppose $a \in A$ satisfies $\alpha_{n-1}(a) = a$. Then $A' = (A,\bracket',\alpha')$, in which
\[
[x_{1,n-1}]' = [x_{1,n-1},a] \quad\text{and}\quad
\alpha' = (\alpha_1,\ldots,\alpha_{n-2}),
\]
is an $(n-1)$-ary Hom-Maltsev algebra.  Moreover, if $A$ is multiplicative, then so is $A'$.
\end{theorem}

\begin{proof}
This follows from Lemma \ref{lem:reduction} because $[a,y_{2,n-1},a]=0$ by the anti-symmetry of $\bracket$.
\end{proof}

Using Theorem \ref{thm:reduction} repeatedly, we obtain the following reduction result from $n$-ary Hom-Maltsev algebras to $(n-k)$-ary Hom-Maltsev algebras.

\begin{theorem}
\label{thm2:reduction}
Let $(A,\bracket,\alpha=(\alpha_1,\ldots,\alpha_{n-1}))$ be an $n$-ary Hom-Maltsev algebra with $n \geq 3$.  Suppose $a_1,\ldots,a_k \in A$ with $k \leq n-2$ satisfy $\alpha_{n-i}(a_i) = a_i$ for each $i \in \{1,\ldots,k\}$.  Then
\begin{equation}
\label{ak}
A^k = (A,\bracket^{(k)},(\alpha_1,\ldots,\alpha_{n-k-1})),
\end{equation}
is an $(n-k)$-ary Hom-Maltsev algebra, where
\[
[x_{1,n-k}]^{(k)} = [x_{1,n-k},a_k,a_{k-1},\ldots,a_1]
\]
for all $x_1,\ldots,x_{n-k} \in A$.  Moreover, if $A$ is multiplicative, then so is $A^k$.
\end{theorem}

The condition $\alpha_{n-i}(a_i) = a_i$ holds automatically if $\alpha_{n-i} = Id$.  Therefore, we have the following special case of Theorem \ref{thm2:reduction}.

\begin{corollary}
\label{cor1:reduction}
Let $(A,\bracket,\alpha=(\alpha_1,\ldots,\alpha_{n-1}))$ be an $n$-ary Hom-Maltsev algebra with $n \geq 3$ such that $\alpha_{n-i} = Id$ for $i \in \{1,\ldots,k\}$ for some $k \leq n-2$.  Then for any elements $a_1,\ldots,a_k \in A$, $A^k$ in \eqref{ak} is an $(n-k)$-ary Hom-Maltsev algebra.
\end{corollary}

Combining Corollary \ref{cor3:hommaltsev} and Theorem \ref{thm:reduction}, we have the following method of constructing a Hom-Maltsev algebra from a composition algebra.

\begin{corollary}
\label{cor2:reduction}
Let $(A,\ast)$ be a composition algebra with involution $\ast$ and $\alpha \colon A \to A$ be an algebra morphism that commutes with $\ast$ and preserves the multiplicative identity of $A$.  Suppose $a \in A$ satisfies $\alpha(a) = a$.  Then
\[
M(A)_\alpha' = (A,[,]_\alpha',\alpha)
\]
is a multiplicative binary Hom-Maltsev algebra, where
\[
[x,y]_\alpha' = \left\{\alpha(xy^*) - \langle x,y\rangle\right\}a - \langle y,a\rangle \alpha(x) + \langle x,a\rangle\alpha(y).
\]
for all $x,y \in A$ and $\langle x,y\rangle$ is the symmetric bilinear form in \eqref{bform}.
\end{corollary}

With $a$ being the multiplicative identity $1$ of $A$, we have the following special case of Corollary \ref{cor2:reduction}.

\begin{corollary}
\label{cor3:reduction}
Let $(A,\ast)$ be a composition algebra with involution $\ast$ and $\alpha \colon A \to A$ be an algebra morphism that commutes with $\ast$ and preserves the multiplicative identity $1$ of $A$. Then
\[
M(A)_\alpha' = (A,[,]_\alpha',\alpha)
\]
is a multiplicative binary Hom-Maltsev algebra, where
\[
[x,y]_\alpha' = \alpha(xy^*) - \langle x,y\rangle - \langle y,1\rangle \alpha(x) + \langle x,1\rangle\alpha(y).
\]
for all $x,y \in A$.
\end{corollary}

%%%%%%%%%%%%%%%%
\begin{example}
\label{ex2:oct}
Recall from Example \ref{ex:octonion} the octonion algebra $\oct$, which is an eight-dimensional composition algebra with involution given by conjugation.  Let $\alpha \colon \oct \to \oct$ be any algebra morphism that commutes with conjugation and preserves the multiplicative identity $e_0$, such as the map in \eqref{octaut}.  The multiplicative binary Hom-Maltsev algebra $M(\oct)'_\alpha = (\oct,[,]'_\alpha,\alpha)$ in Corollary \ref{cor3:reduction} is given by
\[
\begin{split}
[x,y]'_\alpha &= \alpha(x\ybar) - \langle x,y\rangle - \Real(y)\alpha(x) + \Real(x)\alpha(y),
\end{split}
\]
where
\[
\Real(x) = \langle x,e_0\rangle = b_0e_0
\]
if $x = \sum_{i=0}^7 b_ie_i$ with $b_i \in \bk$.
\qed
\end{example}
%%%%%%%%%%%%%%%%

%%==============%%
%%              %%
%%  References  %%
%%              %%
%%==============%%

\end{document}